\documentclass{amsart}

\usepackage[english]{babel}
\usepackage{hyperref}

\usepackage[utf8]{inputenc}
\usepackage[T1]{fontenc}
\usepackage{pifont}

\usepackage[margin=3cm]{geometry}

\usepackage{amsmath}
\usepackage{amssymb}
\usepackage{amsthm}
\usepackage{amscd}
\usepackage{mathrsfs}
\usepackage[abbrev,backrefs]{amsrefs}

\usepackage{stix}

\usepackage{tikz-cd}

\usepackage{graphicx}

\usepackage {cancel}

\numberwithin{equation}{section}

\newtheorem{theorem}{Theorem}[section]	
\newtheorem{proposition}[theorem]{Proposition}	
\newtheorem{corollary}[theorem]{Corollary}	
\newtheorem{lemma}[theorem]{Lemma}
\newtheorem{definition}[theorem]{Definition}

\newtheorem{remark}[theorem]{Remark}

\newcommand{\Z}{\mathbb{Z}}

\newcommand{\C}{\mathscr{C}}
\newcommand{\T}{\mathscr{T}}
\newcommand{\F}{\mathscr{F}}
\newcommand{\z}{\mathscr{Z}}
\newcommand{\n}{\mathscr{N}}

\newcommand{\E}{\mathscr{E}}

\newcommand{\M}{\mathscr{M}}

\begin{document}

\title{\textsc{Torsion theories and coverings of $V$-groups}}
\author{Aline Michel}
\thanks{The author's research is funded by a FRIA doctoral grant of the \emph{Communaut\'e française de Belgique}}
\email{aline.michel@uclouvain.be}
\address{Universit{\'e} Catholique de Louvain, Institut de Recherche en Math{\'e}matique et Physique, Chemin du Cyclotron 2, 1348 Louvain-la-Neuve, Belgium}
\date{}
\begin{abstract}
For a commutative, unital and integral quantale $V$, we generalize to $V$-groups the results developed by Gran and Michel for preordered groups. We first of all show that, in the category $V$-$\mathsf{Grp}$ of $V$-groups, there exists a torsion theory whose torsion and torsion-free subcategories are given by those of indiscrete and separated $V$-groups, respectively. It turns out that this torsion theory induces a monotone-light factorization system that we characterize, and it is then possible to describe the coverings in $V$-$\mathsf{Grp}$. We next classify these coverings as internal actions of a Galois groupoid. Finally, we observe that the subcategory of separated $V$-groups is also a torsion-free subcategory for a pretorsion theory whose torsion subcategory is the one of symmetric $V$-groups. As recently proved by Clementino and Montoli, this latter category is actually not only coreflective, as it is the case for any torsion subcategory, but also reflective. 
\end{abstract}

\maketitle

\section{Introduction}

In the paper \cite{GM}, torsion theories and coverings have been studied in the category $\mathsf{PreOrdGrp}$ of preordered groups \cite{Clementino-Martins-Ferreira-Montoli}. A \emph{preordered group} is a group $(G,+,0)$ endowed with a preorder $\le$ which is compatible with the addition law $+$ of the group $G$: for any $a,b,c,d \in G$, $a \le c$ and $b \le d$ implies that $a + b \le c + d$. In other words, a preordered group is given by a group and a relation satisfying some additional conditions (the reflexivity and the transitivity of the given relation as well as its compatibility with the group structure). Now a relation $R$ on a set $X$ can be seen as a function $r \colon X \times X \rightarrow 2 = \{\top,\bot\}$ where, for any $(x,x') \in X \times X$, $r(x,x') = \top$ if $(x,x') \in R$ and $r(x,x') = \bot$ if $(x,x') \notin R$. This suggests to consider a generalization of the notion of relation: for a commutative and unital \emph{quantale} $V$ (that is, $V$ is a complete lattice $(V, \wedge, \vee, \top, \bot)$ equipped with a binary operation $\otimes$ which is commutative, distributes over arbitrary joins, and has a unit $k$), we could consider \emph{$V$-relations}, a $V$-relation $r \colon X \nvrightarrow X$ on a set $X$ being a function $r \colon X \times X \rightarrow V$.

Based on this notion of $V$-relation, a generalization of the category $\mathsf{PreOrdGrp}$ of preordered groups has been studied in the article \cite{CM}. A \emph{$V$-group} is a group $(X,+,0)$ endowed with a $V$-relation $a \colon X \nvrightarrow X$ satisfying two properties (called the reflexivity and the transitivity axioms), and which is invariant by shifting: for any $x,x',x'' \in X$, $a(x',x'') = a(x' + x,x'' + x)$. In this sense, a preordered group is then a particular case of $V$-group: it is a $2$-group (for the quantale $2 = \{\top,\bot\}$). There are many other very interesting examples of $V$-groups, such as the \emph{Lawvere metric groups}, the \emph{Lawvere ultrametric groups}, the \emph{probabilistic metric groups}, etc \cite{Law}.  

The aim of this paper consists in generalizing to any category of $V$-groups (up to an additional assumption on the quantale $V$) the results on torsion theories and coverings developed in \cite{GM} in the particular setting of preordered groups. We start the article with some prerequisites needed to understand its content. Among other things, we present a summary of the categorical properties of $V$-groups (as described in \cite{CM}), and next we recall the notions of \emph{normal categories}, of \emph{effective descent morphisms}, of \emph{torsion} and \emph{pretorsion theories}, of \emph{factorization systems} and of \emph{coverings} (in the sense of categorical Galois theory).

In the next section, we prove our first result: when $V$ is an \emph{integral} quantale (i.e. when the unit $k$ of $\otimes$ coincides with the top element $\top$), there is a torsion theory in the category $V$-$\mathsf{Grp}$ of $V$-groups. The torsion subcategory $V$-$\mathsf{Grp_{ind}}$ is the full subcategory of $V$-$\mathsf{Grp}$ whose objects are $V$-groups $(X,a)$ such that $a(x,x') = \top$ for any $x,x' \in X$, while the torsion-free subcategory, denoted by $V$-$\mathsf{Grp_{sep}}$, contains all $V$-groups $(X,a)$ satisfying the separation axiom: for any $x,x' \in X$,
\[a(x,x') \geq k \qquad \text{and} \qquad a(x',x) \geq k \qquad \implies \qquad x=x'.\]
Of course, whenever $V$ is integral, this axiom reduces to
\[a(x,x') = \top = a(x',x) \qquad \implies \qquad x=x'.\] 
The objects in $V$-$\mathsf{Grp_{ind}}$ and in $V$-$\mathsf{Grp_{sep}}$ are called \emph{indiscrete} and \emph{separated} $V$-groups, respectively. As a direct consequence (up to some results mentioned in the Background section), we get that $V$-$\mathsf{Grp_{sep}}$ is a \emph{(normal epi)-reflective} subcategory and that $V$-$\mathsf{Grp_{ind}}$ is \emph{(normal mono)-coreflective} in $V$-$\mathsf{Grp}$. The above mentioned torsion theory also induces a factorization system $(\E,\M)$ for which we give a fairly simple description. The class $\M$ of morphisms in $V$-$\mathsf{Grp}$ corresponds to the \emph{trivial coverings} (in the sense of categorical Galois theory). We then prove two intermediate propositions. In one of them, we build, for any $(X,a) \in \text{$V$-}\mathsf{Grp}$, a separated $V$-group $(Y,b)$ as well as an effective descent morphism
\begin{equation} \label{effective descent morphism}
f \colon (Y,b) \rightarrow (X,a)
\end{equation} 
in the category of $V$-groups. Thanks to these two propositions, it is possible to conclude that $(\E',\M^*)$ is a \emph{monotone-light} factorization system \cite{Carboni-Janelidze-Kelly-Pare}, where $\M^*$ is actually the class of coverings in $V$-$\mathsf{Grp}$. We are moreover able to characterize the morphisms in both classes. In particular, when $V$ is an integral quantale, the coverings in $V$-$\mathsf{Grp}$ are given by the morphisms whose kernel is a separated $V$-group.

In section 4, we then make two observations related to the results of section 3. We first notice that it is possible to classify the coverings in $V$-$\mathsf{Grp}$ in terms of $\mathsf{Gal}(f)$-actions, where $f$ is the special effective descent morphism as in \eqref{effective descent morphism} constructed previously and $\mathsf{Gal}(f)$ the \emph{Galois groupoid} associated with $f$. This can be done by using the notion of \emph{locally semisimple covering} introduced by Janelidze, M\'{a}rki and Tholen (see Theorem 3.1 in \cite{JMT}). Secondly, we observe that $V$-$\mathsf{Grp_{sep}}$ is a torsion-free subcategory not only for a torsion theory but also for a pretorsion theory. In this case, the torsion subcategory is given by the full subcategory $V$-$\mathsf{Grp_{sym}}$ of \emph{symmetric} $V$-groups whose objects are $V$-groups $(X,a)$ satisfying the symmetry axiom: for any $x,x' \in X$, $a(x,x') = a(x',x)$. Note that, contrary to the above, for this last result, we do not need to assume that the quantale $V$ is integral. From this pretorsion theory, we deduce that $V$-$\mathsf{Grp_{sym}}$ is coreflective in $V$-$\mathsf{Grp}$. This has already been mentioned in \cite{CM}, where it is also proven that $V$-$\mathsf{Grp_{sym}}$ is reflective.

\section{Background}

\subsection*{The category $V$-$\mathsf{Grp}$ of $V$-groups}

We dedicate this first section to the notion of \emph{$V$-group}: among other things, we define such a notion through the concepts of \emph{$V$-relation} and \emph{$V$-category}, we give some examples of $V$-groups (including the example of preordered groups \cite{Clementino-Martins-Ferreira-Montoli}), and we describe some limits and colimits as well as the different kinds of epimorphisms and monomorphisms, and the short exact sequences in this special setting. We then state some properties about the category of V-groups which will be useful for our work. For this part, we follow the recent paper \cite{CM} where the reader will find more details. 

Consider $V$ a commutative and unital \emph{quantale}, i.e. $V$ is a complete lattice (with bottom element $\bot$ and top element $\top$) endowed with an associative and commutative tensor product $\otimes$ which has a unit $k$ and which preserves arbitrary joins: for any $v \in V$ and any family $\{u_i\}_{i \in I}$ in $V$,
\[v \otimes \left( \bigvee_{i \in I} u_i \right) = \bigvee_{i \in I} \left( v \otimes u_i \right) \qquad \qquad \text{and} \qquad \qquad v \otimes \bot = \bot.\]
We assume moreover that arbitrary joins distribute over finite meets so that, as a lattice, $V$ is a frame. We suppose also that the quantale $V$ is \emph{non-trivial}, i.e. that $\bot \neq \top$.

\begin{remark}
\emph{For part of our developments, we will in addition assume that $V$ is an \emph{integral} quantale, i.e. that $k = \top$ in $V$.}
\end{remark}

In order to understand the notion of \emph{$V$-group}, we first need to define the concepts of \emph{$V$-relation} and of \emph{$V$-category}.

A \emph{$V$-relation} $r : X \nvrightarrow Y$ from a set $X$ to a set $Y$ is a function $r \colon X \times Y \rightarrow V$. As for ordinary relations, a $V$-relation $r \colon X \nvrightarrow Y$ can be composed with a $V$-relation $s \colon Y \nvrightarrow Z$, via a “matricial multiplication”: for any $x \in X$ and any $z \in Z$,
\[(s \cdot r)(x,z) = \bigvee_{y \in Y} r(x,y) \otimes s(y,z),\]
and this defines a $V$-relation $s \cdot r \colon X \nvrightarrow Z$. The identity for the composition is given by the $V$-relation $1_X \colon X \nvrightarrow X$ defined, for any $x,x' \in X$, by
\[1_X(x,x') = \left\{
\begin{array}{lll}
k & \text{if} & x = x'  \\
\bot & \text{if} & x \neq x'.
\end{array}
\right.\]
The behaviour of $V$-relations under composition explains the terminology of \emph{$V$-matrices} we sometimes use instead of that of $V$-relations (see for instance \cite{Seal} for the use of this terminology). Indeed, the above definition of the composition as well as its properties (associativity and identity) make one think of the matrix product.

We then get the category $V$-$\mathsf{Rel}$ of $V$-relations whose objects are sets and whose arrows are $V$-relations. Now there exists a non full, bijective on objects, embedding functor $\mathsf{Set} \rightarrow V \text{-} \mathsf{Rel}$ which associates with any function $f \colon X \rightarrow Y$ the $V$-relation $f \colon X \nvrightarrow Y$ defined, for any $x \in X$ and any $y \in Y$, by
\[f(x,y) = \left\{
\begin{array}{ll}
k & \text{if} \ f(x) = y\\
\bot & \text{otherwise}.
\end{array}
\right.\]
In addition, for any pair $(X,Y)$ of sets, we have an order on the set $V$-$\mathsf{Rel}(X,Y)$ which is induced by the one on the quantale $V$: for $r,r' \colon X \nvrightarrow Y$ two $V$-relations from $X$ to $Y$, we say that $r \le r'$ if and only if, in $V$, $r(x,y) \le r'(x,y)$ for any $x \in X$ and $y \in Y$.

We are now ready for the definition of a \emph{$V$-category}. A $V$-category is a pair $(X,a)$ where $X$ is a set and $a \colon X \nvrightarrow X$ a $V$-relation such that
\[1_X \le a \qquad \qquad \text{and} \qquad \qquad a \cdot a \le a.\]
In pointwise notation, these last two inequalities give the \emph{reflexivity} and \emph{transitivity} axioms:
\begin{itemize}
\item[(R)] $k \le a(x,x)$ $\qquad$ for any $x \in X$;
\item[(T)] $a(x,x') \otimes a(x',x'') \le a(x,x'')$ $\qquad$ for any $x,x',x'' \in X$.
\end{itemize}
Given two $V$-categories $(X,a)$ and $(Y,b)$, a \emph{$V$-functor} $f \colon (X,a) \rightarrow (Y,b)$ is a function $f \colon X \rightarrow Y$ such that $f \cdot a \le b \cdot f$, that is, such that
\[a(x,x') \le b(f(x),f(x')) \qquad \qquad \text{for any} \ x,x' \in X.\]
This gives rise to the category $V$-$\mathsf{Cat}$ of $V$-categories and $V$-functors.

\begin{remark}
\emph{Pairs $(X,a)$ with $X$ a set and $a \colon X \nvrightarrow X$ a $V$-relation satisfying the reflexivity axiom $(R)$ only are said to be \emph{$V$-graphs}.}
\end{remark}

Now, a \emph{$V$-group} $(X,a,+)$ is a $V$-category $(X,a)$ endowed with a group structure 
\[(X,+ \colon X \times X \rightarrow X,- \colon X \rightarrow X,0 \colon 0 \rightarrow X)\]
such that $+ \colon (X,a) \otimes (X,a) \rightarrow (X,a)$ is a $V$-functor, where we define $(X,a) \otimes (X,a)$ by $(X \times X, a \otimes a)$ with $(a \otimes a) ((x_1,x'_1),(x_2,x'_2)) = a(x_1,x_2) \otimes a(x'_1,x'_2)$. A \emph{$V$-homomorphism} is a $V$-functor which is also a group homomorphism. $V$-groups and $V$-homomorphisms define the category $V$-$\mathsf{Grp}$ of $V$-groups.

When we want to check that a given pair $(X,a)$, with $(X,+)$ a group, is a $V$-group, it is sometimes easier to use the following proposition:

\begin{proposition} \cite{CM} \label{equivalent conditions for V-groups}
Let $(X,a)$ be a $V$-graph and $(X,+)$ be a group. Then the following conditions are equivalent:
\begin{enumerate}
\item $+ \colon (X,a) \otimes (X,a) \rightarrow (X,a)$ is a $V$-functor;
\item $(X,a)$ is a $V$-category and $a$ is invariant by shifting, i.e., for any $x,x',x'' \in X$,
\[a(x',x'') = a(x'+x,x''+x).\]
\end{enumerate}
\end{proposition} 

If we consider the quantale $2 = (\{\bot,\top\},\wedge,\top)$ (i.e. $\otimes = \wedge$ and $k = \top$), then any $2$-group is in fact a preordered group, and the converse also holds. Indeed, given a $2$-group $(X,a)$, we can get a preordered group given by $(X,\le)$, where $x \le x'$ in $X$ if and only if $a(x,x') = \top$. Conversely, any preordered group $(X,\le)$ gives rise to a $2$-group $(X,a)$ by defining $a \colon X \nvrightarrow X$, for any $x,x' \in X$, by
\[a(x,x') = \left\{
\begin{array}{ll}
\top & \text{if} \ x \le x'\\
\bot & \text{otherwise.}
\end{array}
\right.\]
As a consequence, the category $2$-$\mathsf{Grp}$ of $2$-groups (and $2$-homomorphisms) is isomorphic to the category $\mathsf{PreOrdGrp}$ of preordered groups (and preorder preserving group homomorphisms), and this latter category is then an example of a category of $V$-groups. Other examples of categories of $V$-groups are given by the categories $\mathsf{MetGrp}$ of \emph{Lawvere (generalized) metric groups} and non expansive group homomorphisms (with $V = ([0, \infty],\geq)$, $\otimes = +$ and $k = 0$), $\mathsf{UMetGrp}$ of \emph{Lawvere (generalized) ultrametric groups} and non expansive group homomorphisms (with $V = ([0,\infty], \geq)$, $\otimes =$ the max for the usual order on real numbers and $k = 0$), $\mathsf{ProbMetGrp}$ of \emph{probabilistic metric groups} (with
\[V = \{ \phi : [0,\infty] \rightarrow [0,1] \ | \ \text{$\phi$ is monotone and} \ \phi (x) = \bigvee_{y < x} \phi(y)\},\]
with the pointwise order, and with $\otimes$ defined by $(\phi \otimes \psi)(u) = \bigvee_{x + y \le u} \phi(x) \times \psi(y)$), etc.

We now mention some interesting full subcategories of the category $V$-$\mathsf{Grp}$ of $V$-groups. First of all, we can consider the subcategory $V$-$\mathsf{Grp_{sym}}$ of \emph{symmetric $V$-groups}, i.e. $V$-groups $(X,a,+)$ which coincide with their dual $(X,a^\circ,+)$, where $a^{\circ}$ is the $V$-relation on $X$ defined, for any $x,x' \in X$ by $a^\circ(x,x') = a(x',x)$. We also have the subcategory $V$-$\mathsf{Grp_{sep}}$ of \emph{separated $V$-groups}, that is of $V$-groups $(X,a,+)$ satisfying the following axiom: for $x,x' \in X$,
\[a(x,x') \geq k \quad \text{and} \quad a(x',x) \geq k \qquad \implies \qquad x = x'.\]
Note that, since $a$ is invariant by shifting, the above property is equivalent to the following : for $x \in X$,
\[a(x,0) \geq k \quad \text{and} \quad a(0,x) \geq k \qquad \implies \qquad x = 0.\]
We borrow this terminology (which, to our knowledge, does not exist in the literature yet) from the more general notion of separated $V$-category (mentioned, for instance, in the Appendix of \cite{HN} and in \cite{HR}). Note that, for $V = ([0,\infty],\geq)$ (with $\otimes = +$), symmetric $V$-groups coincide with the usual symmetric Lawvere metric groups (which are Lawvere metric groups $(X,d)$ with $d(x,x') = d(x',x)$ for any $x,x' \in X$), while separated $V$-groups correspond to the usual separated Lawvere metric groups (i.e. those Lawvere metric groups $(X,d)$ satisfying the separation axiom: $d(x,x') = 0 = d(x',x) \implies x=x'$). We can next consider $V$-groups $(X,a,+)$ with the \emph{indiscrete} $V$-category structure: $a(x,x') = \top$ for any $x,x' \in X$. On the contrary, we can also consider these $V$-groups with the \emph{discrete} $V$-category structure, meaning that
\[a(x,x') = \left\{
\begin{array}{ll}
k & \text{if} \ x = x'\\
\bot & \text{otherwise.}
\end{array}
\right.\]
These two structures give rise to the categories $V$-$\mathsf{Grp_{ind}}$ of indiscrete $V$-groups and $V$-$\mathsf{Grp_{dis}}$ of discrete $V$-groups, respectively, which are both full subcategories of $V$-$\mathsf{Grp}$.

The category $V$-$\mathsf{Grp}$ of $V$-groups is complete and cocomplete. Before describing the limits and some colimits in this particular setting, we first remark that the initial object in $V$-$\mathsf{Grp}$ is given by the $V$-group $(\{*\},\kappa)$ where $\kappa(*,*) = k$ while the terminal object is $(\{*\},c)$ where $c(*,*) = \top$, so that the category $V$-$\mathsf{Grp}$ of $V$-groups is pointed if and only if $k = \top$, i.e. $V$ is integral.

In $V$-$\mathsf{Grp}$, the binary product of two $V$-groups $(X,a)$ and $(Y,b)$ is given by $(X \times Y,a \wedge b)$, where
\[(a \wedge b)((x_1,y_1),(x_2,y_2)) = a(x_1,x_2) \wedge b(y_1,y_2),\]
for any $(x_1,y_1), (x_2,y_2) \in X \times Y$. The equalizer of a pair
\begin{equation} \label{diag equalizer}
\begin{tikzcd}
f,g \colon (X,a) \arrow[r,shift left] \arrow[r,shift right]
& (Y,b)
\end{tikzcd}
\end{equation}
of arrows in $V$-$\mathsf{Grp}$ is given by
\begin{tikzcd}
(E,\tilde{a}) \arrow[r,"e"]
& (X,a)
\end{tikzcd}
where $(E,e)$ is the equalizer of $f$ and $g$ in the category $\mathsf{Grp}$ of groups and $\tilde{a}$ is the $V$-category structure induced by the one of $X$: $\tilde{a} = e^\circ \cdot a \cdot e$ or, equivalently, $\tilde{a}(x,x') = a(e(x),e(x'))$ for any $x,x' \in E$. So the pullback of two arrows $f \colon (X,a) \rightarrow (Z,c)$ and $g \colon (Y,b) \rightarrow (Z,c)$ is $((X \times_Z Y,a \wedge b),\pi_1,\pi_2)$, where $(X \times_Z Y,\pi_1,\pi_2)$ is the pullback of $f$ and $g$ in the category $\mathsf{Grp}$ of groups. In particular, the kernel of $f$ is given by 
\begin{tikzcd}
(K,\tilde{a}) \arrow[r,"k"]
& (X,a)
\end{tikzcd}
where $(K,k)$ is the kernel of $f$ in $\mathsf{Grp}$ and $\tilde{a}$ is the $V$-category structure induced by the one of $X$.

It is also possible to describe the colimits in $V$-$\mathsf{Grp}$. The coequalizer of two parallel arrows $f$ and $g$ as in \eqref{diag equalizer} is obtained by computing the coequalizer $(Q,q)$ of $f$ and $g$ in $\mathsf{Grp}$ and then by equipping the group $Q$ with the final $V$-category structure $\bar{b} = q \cdot b \cdot q^\circ$, that is
\[\bar{b}(z_1,z_2) = \bigvee_{q(y_i) = z_i} b(y_1,y_2)\]
for any $z_1, z_2 \in Q$. In particular, the cokernel of $f$ is given by
\begin{tikzcd}
(Y,b) \arrow[r,"q"] 
& (Q,\bar{b})
\end{tikzcd}
where $(Q,q)$ is the cokernel of $f$ in $\mathsf{Grp}$ and $\bar{b}$ is the final V-category structure. The description of coproducts in $V$-$\mathsf{Grp}$ is more complex but it will not be needed for our paper so that we do not give details about it here and refer the interested reader to \cite{CM}.

Let us now study the different kinds of epimorphisms and monomorphisms which will be of interest for our work. An arrow $f \colon (X,a) \rightarrow (Y,b)$ in $V$-$\mathsf{Grp}$ is an epimorphism if and only if it is surjective. It is a regular epimorphism whenever it is in addition final in the sense that $b = f \cdot a \cdot f^\circ$. Moreover, the category of $V$-groups has the particularity that any regular epimorphism is a normal epimorphism (whenever $V$ is integral). The morphism $f$ is a monomorphism if and only if it injective, and it is a normal monomorphism whenever it is a normal monomorphism in $\mathsf{Grp}$ and $a(x,x') = b(f(x),f(x'))$ for any $x,x' \in X$. 

In the next proposition we gather the information from \cite{CM} which is useful to describe short exact sequences in the category $V$-$\mathsf{Grp}$ of $V$-groups:

\begin{proposition} \label{kernel-cokernel}
Consider, in $V$-$\mathsf{Grp}$, a pair of composable arrows as in the following diagram:
\begin{equation} \label{diag kernel-cokernel}
\begin{tikzcd}
(X,a) \arrow[r,"k"] 
& (Y,b) \arrow[r,"f"]
& (Z,c).
\end{tikzcd}
\end{equation}
Then:
\begin{itemize}
\item the morphism $k$ is the kernel of $f$ if and only if $k$ is the kernel of $f$ in $\mathsf{Grp}$ and $a(x,x') = b(k(x),k(x'))$ for any $x,x' \in X$ (i.e. $a = k^\circ \cdot b \cdot k$);
\item the morphism $f$ is the cokernel of $k$ if and only if $f$ is the cokernel of $k$ in $\mathsf{Grp}$ and $c(z_1,z_2) = \bigvee_{f(y_i) = z_i} b(y_1,y_2)$ for any $z_1,z_2 \in Z$ (i.e. $c = f \cdot b \cdot f^\circ$);
\item the sequence \eqref{diag kernel-cokernel} is a short exact sequence in $V$-$\mathsf{Grp}$ if and only if
\begin{center}
\begin{tikzcd}
0 \arrow[r]
& X \arrow[r,"k"]
& Y \arrow[r,"f"]
& Z \arrow[r]
& 0
\end{tikzcd}
\end{center}
is a short exact sequence in $\mathsf{Grp}$, $a(x,x') = b(k(x),k(x'))$ for any $x,x' \in X$ (i.e. $a = k^\circ \cdot b \cdot k$) and $c(z_1,z_2) = \bigvee_{f(y_i) = z_i} b(y_1,y_2)$ for any $z_1,z_2 \in Z$ (i.e. $c = f \cdot b \cdot f^\circ$).
\end{itemize}
\end{proposition}

We end this section by stating some properties which will turn out to be useful later on. The category $V$-$\mathsf{Grp}$ is first of all regular, but not Barr-exact \cite{Barr}, as observed in Proposition 4.3 in \cite{CM}. In particular, every $V$-homomorphism $f \colon (X,a) \rightarrow (Y,b)$ factorizes as
\begin{center}
\begin{tikzcd}
(X,a) \arrow[rr,"f"] \arrow[dr,two heads,"e"']
& & (Y,b)\\
& (f(X),c) \arrow[ur,tail,"m"'] &
\end{tikzcd}
\end{center}
with $e$ a regular epimorphism in $V$-$\mathsf{Grp}$, $m$ a monomorphism in $V$-$\mathsf{Grp}$ and $c = e \cdot a \cdot e^{\circ}$.
It is moreover \emph{normal} (in the sense of the following section) whenever the quantale $V$ is integral, since any regular epimorphism is then a normal epimorphism, and \emph{effective descent morphisms} (see below for a short explanation) exactly coincide with regular epimorphisms. The detailed proofs of these assertions are due to Clementino and Montoli \cite{CM}.

\begin{remark}
\emph{In order to check the pullback-stability of regular epimorphisms in the category $V$-$\mathsf{Grp}$, Clementino and Montoli used the (above mentioned) assumption that, in the quantale $V$, arbitrary joins distribute over finite meets.}
\end{remark}

\subsection*{Normal categories}

A finitely complete category $\C$ is said to be \emph{normal} \cite{JZ} when
\begin{enumerate}
\item $\C$ has a zero object, denoted by $0$;
\item any arrow $f \colon A \rightarrow B$ in $\C$ factorizes as a normal epimorphism (i.e. a cokernel) followed by a monomorphism;
\item normal epimorphisms are stable under pullbacks: in a pullback diagram
\begin{center}
\begin{tikzcd}
E \times_{B} A \arrow[r,"\pi_{2}"] \arrow[d,"\pi_{1}"']
& A \arrow[d,"f"]\\
E \arrow[r,"p"']
& B
\end{tikzcd}
\end{center}
$\pi_2$ is a normal epimorphism whenever $p$ is a normal epimorphism.
\end{enumerate} 
Categories of groups, abelian groups, rings and Lie algebras are all examples of normal categories, as well as many other algebraic categories. Any semi-abelian category and likewise any homological category is normal. As examples of semi-abelian categories, we can mention the categories of cocommutative Hopf algebras over a field \cite{GSV}, of compact groups \cite{BC} and of $C^*$-algebras \cite{GR}. In connection with the present article, we also have that the category $V$-$\mathsf{Grp}$ of $V$-groups is normal whenever $V$ is an integral quantale (see \cite{CM}). In particular, $\mathsf{PreOrdGrp}$ \cite{Clementino-Martins-Ferreira-Montoli}, $\mathsf{MetGrp}$, $\mathsf{UMetGrp}$ and $\mathsf{ProbMetGrp}$ are normal categories. This information will be crucial for this work.

We now state two fundamental properties of normal categories:

\begin{lemma} \cite{BJ} \label{property-normal}
Let $\C$ be a normal category.
\begin{enumerate}
\item Any regular epimorphism $f \colon A \rightarrow B$ is the cokernel of its kernel, so that the pair $(\ker(f),f)$ forms a short exact sequence in $\C$:
\begin{center}
\begin{tikzcd} [column sep=large]
0 \arrow[r]
& \mathsf{Ker}(f) \arrow[r,"\ker(f)"] 
& A \arrow[r,two heads,"f"] 
& B  \arrow[r]
& 0.
\end{tikzcd}
\end{center}
\item Given a commutative diagram of short exact sequences in $\C$
\begin{equation}\label{generic-ses}
\begin{tikzcd}
0 \arrow[r]
& A \arrow[r,"\kappa"] \arrow[d,,"a"']
& B \arrow[r,"f"] \arrow[d,"b"]
& C \arrow[d,"c"] \arrow[r]
& 0\\
0 \arrow[r]
& A' \arrow[r,"\kappa'"']
& B' \arrow[r,"f'"']
& C' \arrow[r]
& 0
\end{tikzcd}
\end{equation}
the left-hand square is a pullback if and only if the arrow $c$ is a monomorphism.
\end{enumerate}
\end{lemma}

\subsection*{Effective descent morphisms}

We now briefly recall the concept of \emph{effective descent morphism} and refer the interested reader to \cite{JST}, for instance, for a more complete presentation of the subject.

Consider a morphism $p \colon E \rightarrow B$ in a category $\C$ with pullbacks. We write $p^* \colon \C \downarrow B \rightarrow \C \downarrow E$ for the induced pullback functor along $p$, with $\C \downarrow B$ and $\C \downarrow E$ the usual slice categories. We say that the morphism $p \colon E \rightarrow B$ is an \emph{effective descent morphism} when the pullback functor $p^*$ is monadic. Now, there exists a different (and equivalent) way of defining such a morphism, which is expressed in terms of \emph{internal actions}. Let us remind this notion before introducing the alternative definition of effective descent morphisms.

Consider an internal category $\mathcal C$ in $\C$, represented by a diagram of the form 
$$\begin{tikzcd}
 {\mathcal C}: & C_{1} \times_{C_{0}} C_{1} \arrow[r,shift left=2.3,"p_{1}"] \arrow[r,shift right=2.3,"p_{2}"'] \arrow[r,"m" description]
& C_{1} \arrow[r,shift left=2.3,"d"] \arrow[r,shift right=2.3,"c"'] 
& C_{0}, \arrow[l,"s"' description]
\end{tikzcd}$$
where $C_{0}$ is the ``object of objects'', $C_{1}$ the ``object of arrows'', and $C_{1} \times_{C_{0}} C_{1}$ the ``object of composable pairs of arrows'' (see \cite{Borceux}, for instance, for more information about internal categories).
A(n internal) \emph{$\mathcal{C}$-action} is a triple $(A_{0},\pi,\xi)$ as in the diagram
\begin{center}
\begin{tikzcd}
C_{1} \times_{C_{0}} A_{0} \arrow[r,"\xi"]
& A_{0} \arrow[r,"\pi"]
& C_{0},
\end{tikzcd}
\end{center}
where $C_{1} \times_{C_{0}} A_{0}$ is the pullback of the ``domain'' arrow $d: C_{1} \rightarrow C_{0}$ and $\pi : A_{0} \rightarrow C_{0}$, making the following diagram commute:
\begin{center}
\begin{tikzcd}
C_{1} \times_{C_{0}} C_{1} \times_{C_{0}} A_{0} \arrow[r,"1 \times \xi"] \arrow[d,"m \times 1"']
& C_{1} \times_{C_{0}} A_{0} \arrow[d,"\xi"']
& A_{0} \arrow[l,"{(s\cdot \pi,1)}"'] \arrow[dl,equal]\\
C_{1} \times_{C_{0}} A_{0} \arrow[r,"\xi"'] \arrow[d,"\pi_{1}"']
& A_{0} \arrow[d,"\pi"] 
& \\
C_{1} \arrow[r,"c"']
& C_{0} 
&
\end{tikzcd}
\end{center}
where $\pi_{1} : C_{1} \times_{C_{0}} A_{0} \rightarrow C_{1}$ is the first projection in the pullback of $d$ and $\pi$. Given two internal $\mathcal{C}$-actions $(A_0,\pi,\xi)$ and $(A'_0,\pi',\xi')$, a morphism $(A_0,\pi,\xi) \rightarrow (A'_0,\pi',\xi')$ of $\mathcal{C}$-actions \cite{JT97} is given by a pair $(\phi_0,\phi_1)$ of morphisms in $\C$ such that the diagram below commutes:
\begin{center}
\begin{tikzcd}
C_1 \times_{C_0} A_0 \arrow[rr,"\xi"] \arrow[dr,dotted,"\phi_1"] \arrow[dd,"\pi_1"']
& & A_0 \arrow[dr,dotted,"\phi_0"] \arrow[dd,near end,"\pi"]
& \\
& C_1 \times_{C_0} A'_0 \arrow[rr,near start,"\xi'"] \arrow[dl,"\pi'_1"]
& & A'_0 \arrow[dl,"\pi'"] \\
C_1 \arrow[rr,"d"']
& & C_0. &
\end{tikzcd}
\end{center}
Note that the morphism $\phi_1$ is (uniquely) induced by $\phi_0$ and by the universal property of pullbacks. Internal $\mathcal{C}$-actions and their morphisms give rise to a category, the category $\C^{\mathcal{C}}$ of internal $\mathcal{C}$-actions.

One example of internal category is particularly interesting when studying effective descent morphisms. For a morphism $p \colon E \rightarrow B$ in $\C$, consider its kernel pair $(Eq(p),p_1,p_2)$. Then
\begin{center}
\begin{tikzcd}
Eq(p) \arrow[r,shift left = 2,"p_1"] \arrow[r,shift right = 2,"p_2"']
& E \arrow[l,"\Delta" description]
\end{tikzcd}
\end{center} 
is an internal category. To be more precise, it is an \emph{internal (effective) equivalence relation}.

A morphism $p \colon E \rightarrow B$ is called an effective descent morphism if and only if the comparison functor
\[K_p \colon \C \downarrow B \rightarrow \C^{Eq(p)},\]
sending an object $\alpha \colon A \rightarrow B$ of $\C \downarrow B$ to the $Eq(p)$-action
\begin{center}
\begin{tikzcd} [column sep=large]
Eq(p) \times _E (E \times_B A) \arrow[r,"p_1 \times \zeta_2"]
& E \times_B A \arrow[r,"\zeta_1"]
& E 
\end{tikzcd}
\end{center}
where $\zeta_1 \colon E \times_B A \rightarrow E$ and $\zeta_2 \colon E \times_B A \rightarrow A$ are the projections of the pullback $E \times_B A$, is an equivalence of categories. This characterization of effective descent morphisms will turn out to be useful later on.
 
\subsection*{Torsion and pretorsion theories}

One of the main goals of this paper is to find a torsion theory in the category $V$-$\mathsf{Grp}$ of $V$-groups. This is why we devote this section to the concept of \emph{pretorsion theory} in an arbitrary category, which is an extension of that of \emph{torsion theory} (in a normal category). Note that there are many different approaches to non-abelian torsion theories, which are detailed for example in \cite{BG,CDT, EG, JT} (and in the references therein). 

Let $\C$ be an arbitrary category, and let us then recall the notion of \emph{pretorsion theory} in the sense of \cite{FF}. For a thorough study of this topic, we refer to \cite{Facchini-Finocchiaro-Gran} and, for a closely related approach, to \cite{Mantovani,GJ,Clementino}.

Let $\z$ denote a (non-empty) class of objects of $\C$, and then write $\n$ for the class of morphisms in $\C$ that factorize through an object of $\z$. 

Given an arrow $f : A \rightarrow B$ in $\C$, one says that $k : K \rightarrow A$ is a \emph{$\z$-kernel} of $f$ when
\begin{itemize}
\item[$\bullet$] $f \cdot k \in \n$;
\item[$\bullet$] for any morphism $\alpha : X \rightarrow A$ such that $f \cdot \alpha \in \n$, there exists a unique arrow $\phi : X \rightarrow K$ such that $k \cdot \phi = \alpha$.
\end{itemize}
Dually, an arrow $c : B \rightarrow C$ is a \emph{$\z$-cokernel} of $f \colon A \rightarrow B$ when
\begin{itemize}
\item[$\bullet$] $c \cdot f \in \n$;
\item[$\bullet$] for any morphism $\alpha : B \rightarrow X$ such that $\alpha \cdot f \in \n$, there exists a unique arrow $\phi : C \rightarrow X$ such that $\phi \cdot c = \alpha$.
\end{itemize}
Remark that, by definition, any $\z$-kernel is a monomorphism and any $\z$-cokernel is an epimorphism.

\begin{definition}
Let $f : A \rightarrow B$ and $g : B \rightarrow C$ be two composable arrows in $\C$. The sequence
\begin{center}
\begin{tikzcd}
A \arrow[r,"f"]
& B \arrow[r,"g"]
& C
\end{tikzcd}
\end{center}
is said to be a \emph{short $\z$-exact sequence} when $f$ is a $\z$-kernel of $g$ and $g$ is a $\z$-cokernel of $f$.
\end{definition}

Note that, when the category $\C$ is pointed and $\z$ is reduced to the zero object, we recover the classical notions of kernel, cokernel and short exact sequence.

\begin{definition} \label{def pretorsion theory}
A \emph{$\z$-pretorsion theory} in the category $\C$ is given by a pair $(\T,\F)$ of full replete subcategories of $\C$, with $\z = \T \cap \F$, such that:
\begin{enumerate}
\item any morphism in $\C$ from $T \in \T$ to $F \in \F$ belongs to $\n$;
\item for any object $C$ of $\C$ there exists a short $\z$-exact sequence
\begin{center}
\begin{tikzcd}
T \arrow[r,"\epsilon_{C}"]
& C \arrow[r,"\eta_{C}"]
& F
\end{tikzcd}
\end{center}
with $T \in \T$ and $F \in \F$.
\end{enumerate}
\end{definition}

We use the terms \emph{torsion subcategory} and \emph{torsion-free subcategory} to refer to the subcategories $\T$ and $\F$, respectively, by analogy with the terminology of the classical torsion theory $(\mathsf{Ab}_{\text{t.}}, \mathsf{Ab}_{\text{t.f.}})$ in the category $\mathsf{Ab}$ of abelian groups, where $\mathsf{Ab}_{\text{t.}}$ is the category of torsion abelian groups and $\mathsf{Ab}_{\text{t.f.}}$ the category of torsion-free abelian groups.

Note that the short $\z$-exact sequence in Definition \ref{def pretorsion theory} (2) is actually unique up to isomorphism. Indeed, let us assume that we have, for an object $C$ in $\C$, two such short exact sequences
\begin{center}
\begin{tikzcd}
T \arrow[r,"\epsilon_C"] \arrow[d,dotted,"t"']
& C \arrow[r,"\eta_C"] \arrow[d,equal]
& F \arrow[d,dotted,"f"]\\
T' \arrow[r,"\epsilon'_C"']
& C \arrow[r,"\eta'_C"']
& F',
\end{tikzcd}
\end{center}
that is, such that $T, T' \in \T$ and $F, F' \in \F$. Then, since $T \in \T$ and $F' \in \F$, $\eta'_C \cdot \epsilon_C \in \n$ so that, by the universal property of $\z$-kernels, there exists a unique arrow $t \colon T \rightarrow T'$ making the left-hand square commute. Now the universal property of $\z$-kernels applied to $\epsilon_C$ (which is the $\z$-kernel of $\eta_C$) also gives us a unique morphism $s \colon T' \rightarrow T$ such that $\epsilon_C \cdot s = \epsilon'_C$, because $\eta_C \cdot \epsilon'_C \in \n$, and it is easy to check that this morphism $s$ is the inverse of $t$. As a conclusion, $t$ is an isomorphism. Dually, we also get a unique arrow $f \colon F \rightarrow F'$ making the right-hand square commute, and it turns out to be an isomorphism as well. 

Consider now a morphism $\phi \colon C \rightarrow C'$ in the category $\C$, and the two (unique up to isomorphism) short $\z$-exact sequences associated with $C$ and $C'$:
\begin{center}
\begin{tikzcd}
T \arrow[r,"\epsilon_C"] \arrow[d,dotted,"T(\phi)"']
& C \arrow[r,"\eta_C"] \arrow[d,"\phi"]
& F \arrow[d,dotted,"F(\phi)"]\\
T' \arrow[r,"\epsilon_{C'}"']
& C' \arrow[r,"\eta_{C'}"']
& F'.
\end{tikzcd}
\end{center}
Using the same kind of reasoning as above, we find two morphisms $T(\phi) \colon T \rightarrow T'$ and $F(\phi) \colon F \rightarrow F'$ such that $\epsilon_{C'} \cdot T(\phi) = \phi \cdot \epsilon_C$ and $F(\phi) \cdot \eta_C = \eta_{C'} \cdot \phi$. This construction gives rise to two functors, $F \colon \C \rightarrow \F$ and $T \colon \C \rightarrow \T$. The first one is a left adjoint for the inclusion functor $U \colon \F \rightarrow \C$ and, for any $C \in \C$, the $C$-component of the unit of the adjunction $F \dashv U$ is given by the epimorphism $\eta_C \colon C \rightarrow F = UF(C)$ in Definition \ref{def pretorsion theory}. As a consequence, the functor $F$ is an epi-reflector. The dual statement also holds: the functor $T$ is a mono-coreflector. Indeed, $T$ is a right adjoint of the inclusion functor $V \colon \T \rightarrow \C$ and any $C$-component $\epsilon_C \colon T = VT(C) \rightarrow C$ of the counit of the adjunction $V \dashv T$, which corresponds to the arrow $\epsilon_C$ in Definition \ref{def pretorsion theory}, is a monomorphism. 

Now, if the category $\C$ is normal (so in particular pointed), we find the notion of \emph{torsion theory} by considering $\z$ the class with the zero object only:

\begin{definition} \label{def torsion theory}
A \emph{torsion theory} in a normal category $\C$ is given by a pair $(\T,\F)$ of full (replete) subcategories of $\C$ such that:
\begin{enumerate}
\item the only arrow from any $T \in \T$ to any $F \in \F$ is the zero arrow;
\item for any object $C$ of $\C$ there exists a short exact sequence
\begin{center}
\begin{tikzcd}
0 \arrow[r]
& T \arrow[r,"\epsilon_{C}"]
& C \arrow[r,"\eta_{C}"]
& F \arrow[r]
& 0
\end{tikzcd}
\end{center}
whith $T \in \T$ and $F \in \F$.
\end{enumerate}
\end{definition}

In this particular situation, we also find the results mentioned above: the short exact sequence of Definition \ref{def torsion theory} (2) is unique up to isomorphism, and any torsion theory gives rise to two functors $F \colon \C \rightarrow \F$ and $T \colon \C \rightarrow \T$ which are (normal epi)-reflection and (normal mono)-coreflection, respectively.

\subsection*{Monotone-light factorization systems and coverings}

Another goal of this article is to characterize the \emph{coverings} (in the sense of categorical Galois theory) in the category $V$-$\mathsf{Grp}$ of $V$-groups. Let us then recall a few notions and results which will be helpful for our future developments. For this part, we mainly follow \cite{Cassidy-Hebert-Kelly, Carboni-Janelidze-Kelly-Pare, EG, Janelidze}, where the reader will find more information about \emph{(monotone-light) factorization systems} and \emph{coverings}.

In this section, $\C$ will denote an arbitrary category. Consider the particular case where we have a full reflective subcategory $\F$ of $\C$:

\begin{equation}\label{reflection F}
\begin{tikzcd}[column sep = small]
\C \arrow[rr,shift left=0.2cm,"F"] 
& \bot
& \F. \arrow[ll,shift left=0.2cm,"U"]
\end{tikzcd}
\end{equation}

Then something interesting happens when the reflector $F \colon \C \rightarrow \F$ is \emph{semi-left-exact}:

\begin{definition} \cite{Cassidy-Hebert-Kelly}
A reflector $F : \C \rightarrow \F$ as in \eqref{reflection F} is \emph{semi-left-exact} when it preserves pullbacks of the form 
\begin{center}
\begin{tikzcd}
P \arrow[r] \arrow[d]
& U(C) \arrow[d,"U(f)"]\\
B \arrow[r,"\eta_{B}"']
& UF(B),
\end{tikzcd}
\end{center}
where $\eta_{B} : B \rightarrow UF(B)$ is the $B$-component of the unit of the adjunction \eqref{reflection F} and $f : C \rightarrow F(B)$ is an arrow in the subcategory $\F$ of $\C$.
\end{definition}

In fact, a reflection is semi-left-exact if and only if it is \textit{admissible} in the sense of categorical Galois theory \cite{Janelidze} (with respect to the classes of \emph{all} morphisms, as explained in \cite{Carboni-Janelidze-Kelly-Pare}).

Note that there exists a natural property, for a reflector $F \colon \C \rightarrow \F$, that is stronger than being semi-left-exact:

\begin{definition}\cite{Cassidy-Hebert-Kelly} 
A reflector $F : \C \rightarrow \F$ as in \eqref{reflection F} has \emph{stable units} when it preserves pullbacks of the form 
\begin{center}
\begin{tikzcd}
P \arrow[r] \arrow[d]
& C \arrow[d,"f"]\\
B \arrow[r,"\eta_{B}"']
& UF(B)
\end{tikzcd}
\end{center}
where $\eta_{B} : B \rightarrow UF(B)$ is the $B$-component of the unit of the adjunction \eqref{reflection F} and $f : C \rightarrow UF(B)$ is any arrow in the category $\C$.
\end{definition}

In the semi-left-exact context, we then naturally get a \emph{factorization system} $(\E,\M)$ defined as follows \cite{Cassidy-Hebert-Kelly}:
\begin{itemize}
\item[$\bullet$] $\E = \{f \in \C \ | \ F(f) \ \text{is an isomorphism} \}$;
\item[$\bullet$] $\M = \{f \in \C \ | \ \text{the following square \eqref{naturality} is a pullback} \}$:
\begin{equation}\label{naturality}
\begin{tikzcd}
A \arrow[r,"\eta_{A}"] \arrow[d,"f"']
& UF(A) \arrow[d,"UF(f)"]\\
B \arrow[r,"\eta_{B}"'] 
& {UF(B),}
\end{tikzcd}
\end{equation}
where $\eta$ is the unit of the adjunction \eqref{reflection F},
\end{itemize}
and the morphisms in the class $\M$ are called \emph{trivial coverings}.

\begin{remark} \label{remark TT-stable units}
\emph{When we have a torsion theory $(\T,\F)$ in a normal category $\C$, the (induced) reflector $F \colon \C \rightarrow \F$ then gives rise to a factorization system $(\E,\M)$ as defined above since any reflector induced by a torsion theory in a normal category has stable units \cite{Everaert-Gran15}.}
\end{remark}

Given the reflection \eqref{reflection F}, we now consider the following two subclasses of morphisms in $\C$:
\begin{itemize}
\item[$\bullet$] $\E' = \{f \in \C \ | \ \text{the pullback of $f$ along any morphism in $\C$ is in $\E$} \}$;
\item[$\bullet$] $\M^{*} = \{f \in \C \ | \ \text{there exists an effective descent morphism $p$ such that $p^{*}(f)$ is in $\M$} \}$.
\end{itemize}

Morphisms in $\M^{*}$ are called \emph{coverings} (in the sense of categorical Galois theory). As already mentioned, one of the main goals of this paper is to describe these coverings in the category $V$-$\mathsf{Grp}$ of $V$-groups, and also to prove that, in this context, the pair $(\E',\M^*)$ is a \emph{monotone-light factorization system} in the following sense:

\begin{definition}\cite{Carboni-Janelidze-Kelly-Pare} 
A factorization system is said to be \emph{monotone-light} when it is of the form $(\E',\M^{*})$ for some factorization system $(\E,\M)$. 
\end{definition}

\section{Torsion theory and coverings in the category of $V$-groups (for an integral quantale $V$)}

As previously announced, from now on, the commutative and unital quantale $V$ will be assumed to be also integral. Under this additional hypothesis, the reflexivity axiom from the definition of a $V$-category now becomes:
\[a(x,x) = k = \top \qquad \qquad \text{for any} \ x \in X.\]

\subsection*{A torsion theory in $V$-$\mathsf{Grp}$}

We first prove that there is a torsion theory in the normal category $V$-$\mathsf{Grp}$, where the torsion subcategory is the category of indiscrete $V$-groups while the torsion-free subcategory is that of separated $V$-groups.

\begin{proposition} \label{torsion theory}
The pair of full and replete subcategories ($V$-$\mathsf{Grp_{ind}}$,$V$-$\mathsf{Grp_{sep}}$) of $V$-$\mathsf{Grp}$ is a torsion theory in the normal category $V$-$\mathsf{Grp}$.
\end{proposition}

\begin{proof}
Let us first show that the only arrow in $V$-$\mathsf{Grp}$ from an object of $V$-$\mathsf{Grp_{ind}}$ to an object of $V$-$\mathsf{Grp_{sep}}$ is the zero morphism. For this, consider an arrow $f \colon (X,a) \rightarrow (Y,b)$ in $V$-$\mathsf{Grp}$, with $(X,a) \in$ $V$-$\mathsf{Grp_{ind}}$ and $(Y,b) \in$ $V$-$\mathsf{Grp_{sep}}$. Then, since $f$ is a $V$-homomorphism, for any $x \in X$,
\[b(f(x),0) \ge a(x,0) = \top \qquad \implies \qquad b(f(x),0) \geq k\]
and, similarly,
\[b(0,f(x)) \ge a(0,x) = \top \qquad \implies \qquad b(0,f(x)) \geq k,\]
which implies that $f(x) = 0$ since $f(x) \in Y$ with $(Y,b)$ a separated $V$-group, and then $f = 0$.

Let now $(X,a)$ be an object of $V$-$\mathsf{Grp}$, and let us define the subset
\[N_X = \{x \in X | a(0,x) \geq k \ \text{and} \ a(x,0) \geq k\}\]
of $X$. It is a subgroup of $X$ since
\begin{itemize}
\item $0 \in N_X$: $a(0,0) \geq k $;
\item $x \in N_X \implies -x \in N_X$: $a(0,-x) = a(x,0) \geq k$ and $a(-x,0) = a(0,x) \geq k$;
\item $x,y \in N_X \implies x + y \in N_X$: 
\[a(0,x+y) \ge a(0,x) \otimes a(0,y) \geq k \otimes k = k\] 
and
\[a(x+y,0) \ge a(x,0) \otimes a(y,0) \geq k \otimes k = k\]
since $+ \colon (X,a) \otimes (X,a) \rightarrow (X,a)$ is a $V$-functor.
\end{itemize}
This subgroup $N_X$ is actually normal in $X$. Indeed, by invariance of $a$ by shifting, for $x \in X$ and $n \in N_X$, we have that
\[a(0,x+n-x) = a(-x+x,n) = a(0,n) \geq k\]
and
\[a(x+n-x,0) = a(n,-x+x) = a(n,0) \geq k,\]
which implies that $x+n-x \in N_X$, as desired. So the sequence
\begin{tikzcd}
N_X \arrow[r,"k_X"] 
& X \arrow[r,two heads,"\eta_X"]
& X/N_X,
\end{tikzcd}
where $k_X$ is the inclusion of $N_X$ in $X$ and $\eta_X$ the quotient morphism, is a short exact sequence in $\mathsf{Grp}$. Now, we endow $N_X$ with the $V$-category structure $\tilde{a}$ induced by the one of $X$, i.e. $\tilde{a} = k_X^\circ \cdot a \cdot k_X$, and we equip the quotient $X/N_X$ with the final structure $\bar{a}$, that is $\bar{a} = \eta_X \cdot a \cdot \eta_X^\circ$. By Proposition \ref{kernel-cokernel}, the sequence
\begin{equation} \label{SES}
\begin{tikzcd}
(N_X,\tilde{a}) \arrow[r,"k_X"]
& (X,a) \arrow[r,two heads,"\eta_X"]
& (X/N_X,\bar{a})
\end{tikzcd}
\end{equation}
is a short exact sequence in the category $V$-$\mathsf{Grp}$ of $V$-groups. It remains to show that $(N_X,\tilde{a}) \in$ $V$-$\mathsf{Grp_{ind}}$ and that $(X/N_X,\bar{a}) \in$ $V$-$\mathsf{Grp_{sep}}$. We first of all compute that, for any $n,n' \in N_X$,
\[\tilde{a}(n,n') = a(k_X(n),k_X(n')) = a(n,n') = a(0,n'-n) \geq k = \top,\]
since $n' - n \in N_X$ and $V$ is integral, which implies that $\tilde{a}(n,n') = \top$, and this shows that $(N_X,\tilde{a}) \in$ $V$-$\mathsf{Grp_{ind}}$. Assume next that $\bar{a}(y,0) \geq k$ and that $\bar{a}(0,y) \geq k$ for $y \in X/N_X$. Then 
\begin{equation} \label{equation 1}
k \le \bar{a}(y,0) = \bigvee_{\substack{\eta_X(x_1) = y \\ \eta_X(x_2) = 0}} a(x_1,x_2)
\end{equation}
and
\begin{equation} \label{equation 2}
k \le \bar{a}(0,y) = \bigvee_{\substack{\eta_X(x'_1) = 0 \\ \eta_X(x'_2) = y}} a(x'_1,x'_2).
\end{equation}
Now, we observe that, for any $x_1,x_2,x'_1,x'_2 \in X$ such that $\eta_X(x_1) = \eta_X(x'_1)$ and $\eta_X(x_2) = \eta_X(x'_2)$, we have
\begin{align*}
a(x_1,x_2) & = a(x_1 - x'_1 + x'_1,x_2 - x'_2 + x'_2)\\
& \geq a(x_1 - x'_1,0) \otimes a(0,x_2 - x'_2) \otimes a(x'_1,x'_2)\\
& \geq k \otimes k \otimes a(x'_1,x'_2) = a(x'_1,x'_2)
\end{align*}
since $+ \colon (X,a) \otimes (X,a) \rightarrow (X,a)$ is a $V$-functor and since $x_1 - x'_1 \in N_X$ and $x_2 - x'_2 \in N_X$. Symmetrically, we also show that $a(x'_1,x'_2) \geq a(x_1,x_2)$ for such $x_1,x_2,x'_1,x'_2$, and this proves that $a(x_1,x_2) = a(x'_1,x'_2)$ for any $x_1,x_2,x'_1,x'_2 \in X$ such that $\eta_X(x_1) = \eta_X(x'_1)$ and $\eta_X(x_2) = \eta_X(x'_2)$. Since the morphism $\eta_X \colon X \rightarrow X/N_X$ is surjective, there exists an $x \in X$ such that $\eta_X(x) = y$. Equations \eqref{equation 1} and \eqref{equation 2} then imply, by idempotence of the operation $\vee$, that $a(x,0) \geq k$ and that $a(0,x) \geq k$. In other words, $x \in N_X$. This shows that $(X/N_X,\bar{a}) \in$ $V$-$\mathsf{Grp_{sep}}$ since
\[y = \eta_X(x) = 0. \qedhere\]
\end{proof}

\begin{remark} \label{remark torsion theory}
\emph{In the above proof, we did not use the fact that $V$ is an integral quantale when we were showing that $N_X$ is a normal subgroup of $X$ and that $(X/N_X,\bar{a})$ is a separated $V$-group. This observation will be useful for some future developments.}
\end{remark}

As a direct consequence of Proposition \ref{torsion theory}, we get the following result:

\begin{corollary} \hspace*{3cm}
\begin{itemize}
\item The category $V$-$\mathsf{Grp_{sep}}$ is reflective in the category $V$-$\mathsf{Grp}$
\begin{equation} \label{reflection}
\begin{tikzcd}
\text{$V$-}\mathsf{Grp} \arrow[rr,shift left=2,"F"]
& \bot
& \text{$V$-}\mathsf{Grp_{sep}} \arrow[ll,hook',shift left=2,"U"]
\end{tikzcd}
\end{equation}
and each component of the unit $\eta$ of the adjunction (as in \eqref{SES}) is a normal epimorphism.
\item The category $V$-$\mathsf{Grp_{ind}}$ is coreflective in $V$-$\mathsf{Grp}$ and each component of the counit $\kappa$ of the adjunction (as in \eqref{SES}) is a normal monomorphism.
\end{itemize}
\end{corollary}

\begin{proof}
This follows from Proposition \ref{torsion theory} and the (only) Proposition in \cite{JT} (see also \cite{BG}, and \cite{CDT}).
\end{proof}

\subsection*{Monotone-light factorization system and coverings in $V$-$\mathsf{Grp}$}

As recalled in the Background section, by Remark \ref{remark TT-stable units}, the reflection \eqref{reflection} gives rise to a factorization system $(\E,\M)$ since the reflective subcategory $V$-$\mathsf{Grp_{sep}}$ is a torsion-free subcategory in $V$-$\mathsf{Grp}$ (as a consequence of Proposition \ref{torsion theory}). The next proposition characterizes its two classes of morphisms $\E$ and $\M$ as follows:

\begin{proposition}
Consider, in $V$-$\mathsf{Grp}$, the commutative diagram
\begin{center}
\begin{tikzcd}
(N_X,\tilde{a}) \arrow[rr,"k_X"] \arrow[dd,"\phi"']
& & (X,a) \arrow[rr,two heads,"\eta_X"] \arrow[dd,"f"]
& & (X/N_X,\bar{a}) \arrow[dd,"\alpha"]\\
 & (1) & & (2) & \\
(N_Y,\tilde{b}) \arrow[rr,"k_Y"']
& & (Y,b) \arrow[rr,two heads,"\eta_Y"']
& & (Y/N_Y,\bar{b})
\end{tikzcd}
\end{center}
where, as before, $\tilde{a} = k_X^\circ \cdot a \cdot k_X$, $\tilde{b} = k_Y^\circ \cdot b \cdot k_Y$, $\bar{a} = \eta_X \cdot a \cdot \eta_X^\circ$, and $\bar{b} = \eta_Y \cdot b \cdot \eta_Y^\circ$, where $\eta_X$ (respectively $\eta_Y$) is the $(X,a)$-component (respectively the $(Y,b)$-component) of the unit of the adjunction \eqref{reflection}, $\phi$ is the restriction of $f$ to $(N_X,\tilde{a})$ and where we write $\alpha$ for $F(f)$.\\
The adjunction \eqref{reflection} induces a factorization system $(\E,\M)$ in $V$-$\mathsf{Grp}$ where:
\begin{itemize}
\item $f \colon (X,a) \rightarrow (Y,b)$ is in the class $\E$ if and only if the following conditions hold:
\begin{itemize}
\item[(a)] $f^{-1}(N_Y) = N_X$;
\item[(b)] $\tilde{a} = \tilde{b} \wedge a$;
\item[(c)] $\eta_Y \cdot f$ is a regular epimorphism;
\end{itemize}
\item $f \colon (X,a) \rightarrow (Y,b)$ is in the class $\M$ if and only if the restriction $\phi \colon N_X \rightarrow N_Y$ of $f$ to $N_X$ is a group isomorphism and $a = b \wedge \bar{a}$.
\end{itemize}
\end{proposition}

\begin{proof}
\begin{itemize}
\item Assume that $f \colon (X,a) \rightarrow (Y,b)$ is in the class $\E$, so that $\alpha$ is an isomorphism in $V$-$\mathsf{Grp}$. Since $V$-$\mathsf{Grp}$ is a normal category, by Lemma \ref{property-normal}, the fact that $\alpha$ is a monomorphism implies that the square $(1)$ in the above diagram is a pullback, i.e. $f^{-1}(N_Y) = N_X$ and $\tilde{a} = \tilde{b} \wedge a$. Now, knowing that $\alpha$ is a regular epimorphism, we have that the composite $\alpha \cdot \eta_X$ is also a regular epimorphism (since $\eta_X$ is itself a regular epimorphism and the category $V$-$\mathsf{Grp}$ is regular), that is $\eta_Y \cdot f$ is a regular epimorphism by commutativity of the square $(2)$.

Conversely, if $f^{-1}(N_Y) = N_X$ and $\tilde{a} = \tilde{b} \wedge a$, then the square $(1)$ in the above diagram is a pullback and this implies, by Lemma \ref{property-normal}, that $\alpha$ is a monomorphism since $V$-$\mathsf{Grp}$ is a normal category. The assumption $(c)$ then implies that $\alpha \cdot \eta_X$ is a regular epimorphism (by commutativity of the square $(2)$). Knowing that the category $V$-$\mathsf{Grp}$ is regular, it follows that $\alpha$ is a regular epimorphism. As a conclusion, $\alpha$ is an isomorphism, that is $f$ is in the class $\E$.
\item Assume that $f \colon (X,a) \rightarrow (Y,b)$ is in the class $\M$. It means that the square $(2)$ in the above diagram is a pullback. It is then easily seen that $a = b \wedge \bar{a}$ and that the arrow $\phi$ is an isomorphism in the category of V-groups, so in particular in the category of groups.

Conversely, if we suppose that $\phi$ is a group isomorphism, then the square
\begin{center}
\begin{tikzcd}
X \arrow[r,two heads,"\eta_X"] \arrow[d,"f"']
& X/N_X \arrow[d,"\alpha"]\\
Y \arrow[r,two heads,"\eta_Y"']
& Y/N_Y
\end{tikzcd}
\end{center}
is a pullback in the category $\mathsf{Grp}$ of groups since this category is indeed protomodular \cite{Bourn}. By assumption, we also know that $a = b \wedge \bar{a}$ so that the square $(2)$ in the above diagram is a pullback in $V$-$\mathsf{Grp}$. It follows that $f$ is in the class $\M$. \qedhere
\end{itemize}
\end{proof}

In particular, the previous proposition provides us with a description of the trivial coverings (i.e. the $V$-homomorphisms in the class $\M$). In order to find \emph{all} the coverings, we use the following result from \cite{EG} (see also \cite{Carboni-Janelidze-Kelly-Pare}), which not only gives a very simple characterization of the morphisms in the class $\M^*$ but also states that the pair $(\E',\M^*)$ is a monotone-light factorization system.  

\begin{theorem} \label{thmEG}
Let $\C$ be a normal category. Let $(\T,\F)$ be a torsion theory in $\C$ such that, for any normal monomorphism $k : K \rightarrow A$, the monomorphism $k \cdot \epsilon_{K} : T(K) \rightarrow A$ is normal in $\C$, where $\epsilon_{K} : T(K) \rightarrow K$ is the K-component of the counit $\epsilon$ of the coreflection $\C \rightarrow \T$. We write $(\E,\M)$ for the factorization system associated with the reflector $F : \C \rightarrow \F$, which has stable units. \\
If for any object $C$ in $\C$ there is an effective descent morphism $p : F \rightarrow C$ with $F \in \F$, then $({\E}',{\M}^*)$ is a monotone-light factorization system, and moreover
\begin{itemize}
\item[$\bullet$] ${\E}'$ is the class of normal epimorphisms in $\C$ whose kernel is in $\T$;
\item[$\bullet$] ${\M}^*$ is the class of morphisms in $\C$ whose kernel is in $\F$. 
\end{itemize} 
\end{theorem}

It only remains to check two assumptions in order to be allowed to apply Theorem \ref{thmEG} to our context. This is the aim of the following two propositions.

\begin{proposition} \label{V-Grp_sep covers V-Grp}
For any object $(X,a)$ in the category $V$-$\mathsf{Grp}$ of $V$-groups, there exist an object $(Y,b)$ in the subcategory $V$-$\mathsf{Grp_{sep}}$ of separated $V$-groups and an effective descent morphism
\[f \colon (Y,b) \rightarrow (X,a)\]
from $(Y,b)$ to $(X,a)$.
\end{proposition}

\begin{proof}
Let $(X,a) \in$ $V$-$\mathsf{Grp}$. Define $(Y,b)$ in the following way:
\begin{itemize}
\item $Y = \Z \times X$;
\item for any $z,z' \in \Z$ and any $x,x' \in X$,
\[b((z,x),(z',x')) = \left\{
\begin{array}{cl}
a(x,x') & \text{if} \ z < z'\\
\top & \text{if} \ z=z' \ \text{and} \ x=x'\\
\bot & \text{otherwise.}
\end{array}
\right.\]
\end{itemize}
Let us first check that $(Y,b)$ defined in this way is a $V$-category. The reflexivity axiom is easy to verify: for any $(z,x) \in Y$,
\[b((z,x),(z,x)) = \top = k.\]
For the transitivity axiom, we need some developments. We recall that we have to prove the following: for any $(z,x),(z',x'),(z'',x'') \in Y$,
\[b((z,x),(z',x')) \otimes b((z',x'),(z'',x'')) \le b((z,x),(z'',x'')).\]
For this, we consider the different possible cases.
\begin{enumerate}
\item $z < z''$:
\begin{itemize}
\item $z' < z < z''$:
\begin{align*}
b((z,x),(z',x')) \otimes b((z',x'),(z'',x'')) & = \bot \otimes a(x',x'') = \bot\\
& \le a(x,x'') = b((z,x),(z'',x'')).
\end{align*}
\item $z'=z < z''$: If $x=x'$, then
\begin{align*}
b((z,x),(z',x')) \otimes b((z',x'),(z'',x'')) & = \top \otimes a(x',x'') \\
& = a(x,x) \otimes a(x',x'') = a(x,x') \otimes a(x',x'')\\
& \le a(x,x'') = b((z,x),(z'',x''))
\end{align*}
and, if $x \neq x'$, then
\begin{align*}
b((z,x),(z',x')) \otimes b((z',x'),(z'',x'')) & = \bot \otimes a(x',x'') = \bot\\
& \le a(x,x'') = b((z,x),(z'',x'')).
\end{align*}
\item $z < z' < z''$:
\begin{align*}
b((z,x),(z',x')) \otimes b((z',x'),(z'',x'')) & = a(x,x') \otimes a(x',x'')\\
& \le a(x,x'') = b((z,x),(z'',x'')).
\end{align*}
\item $z < z'=z''$: similar to the case $z' = z < z''$.
\item $z < z'' < z'$: similar to the case $z' < z < z''$.
\end{itemize}
\item $z=z''$ and $x=x''$:
\[b((z,x),(z',x')) \otimes b((z',x'),(z'',x'')) \le \top = b((z,x),(z'',x'')).\]
\item $z=z''$ and $x \neq x''$:
\begin{itemize}
\item $z' < z=z''$:
\begin{align*}
b((z,x),(z',x')) \otimes b((z',x'),(z'',x'')) & = \bot \otimes a(x',x'')\\
& = \bot = b((z,x),(z'',x'')).
\end{align*}
\item $z = z' = z''$:
\begin{itemize}
\item If $x = x'$ and $x' = x''$, then $x = x''$, which is a contradiction. So we do not consider this case.
\item If $x = x'$ and $x' \neq x''$, then
\begin{align*}
b((z,x),(z',x')) \otimes b((z',x'),(z'',x'')) & = \top \otimes \bot\\
& = \bot = b((z,x),(z'',x'')).
\end{align*}
\item If $x \neq x'$ and $x' = x''$, then we are in a situation which is symmetrical to the case $x = x'$ and $x' \neq x''$.
\item If $x \neq x'$ and $x' \neq x''$, then
\begin{align*}
b((z,x),(z',x')) \otimes b((z',x'),(z'',x'')) & = \bot \otimes \bot\\
& = \bot = b((z,x),(z'',x'')).
\end{align*}
\end{itemize}
\item $z = z'' < z'$: similar to the case $z' < z = z''$.
\end{itemize}
\item $z > z''$:
\begin{itemize}
\item $z' < z'' < z$:
\begin{align*}
b((z,x),(z',x')) \otimes b((z',x'),(z'',x'')) & = \bot \otimes a(x',x'')\\
& = \bot = b((z,x),(z'',x'')).
\end{align*}
\item $z'=z'' < z$: If $x' = x''$, then
\begin{align*}
b((z,x),(z',x')) \otimes b((z',x'),(z'',x'')) & = \bot \otimes \top\\
& = \bot = b((z,x),(z'',x''))
\end{align*}
and, if $x' \neq x''$, then
\begin{align*}
b((z,x),(z',x')) \otimes b((z',x'),(z'',x'')) & = \bot \otimes \bot\\
& = \bot = b((z,x),(z'',x'')).
\end{align*}
\item $z'' < z' < z$:
\begin{align*}
b((z,x),(z',x')) \otimes b((z',x'),(z'',x'')) & = \bot \otimes \bot\\
& = \bot = b((z,x),(z'',x'')).
\end{align*}
\item $z'' < z'=z$: similar to the case $z' = z'' < z$.
\item $z'' < z < z'$: similar to the case $z' < z'' < z$.
\end{itemize}
\end{enumerate}
All this proves that $(Y,b)$ is a $V$-category. Moreover, $b$ is invariant by shifting, that is, for any $(z,x)$, $(z',x')$, $(z'',x'') \in Y$,
\[b((z',x'),(z'',x'')) = b((z'+z,x'+x),(z''+z,x''+x)).\]
Indeed:
\begin{itemize}
\item if $z' < z''$, then $z' + z < z'' + z$, so that
\[b((z',x'),(z'',x'')) = a(x',x'') = a(x'+ x,x'' + x) = b((z'+z,x'+x),(z''+z,x''+x));\]
\item if $z' = z''$ and $x' = x''$, then $z' + z = z'' + z$ and $x' + x = x'' + x$, so that
\[b((z',x'),(z'',x'')) = \top = b((z'+z,x'+x),(z''+z,x''+x));\]
\item if $z' = z''$ and $x' \neq x''$, then $z' + z = z'' + z$ and $x' + x \neq x'' + x$, so that
\[b((z',x'),(z'',x'')) = \bot = b((z'+z,x'+x),(z''+z,x''+x));\]
\item if $z' > z''$, then $z' + z > z'' + z$, so that
\[b((z',x'),(z'',x'')) = \bot = b((z'+z,x'+x),(z''+z,x''+x)).\]
\end{itemize}
As a conclusion, thanks to Proposition \ref{equivalent conditions for V-groups}, we deduce that $(Y,b)$ is a $V$-group. In particular, $(Y,b) \in$ $V$-$\mathsf{Grp_{sep}}$. Indeed, for $(z,x) \in Y$, if
\[b((z,x),(0,0)) = \top = b((0,0),(z,x)),\]
then $(z,x) = (0,0)$ since
\begin{itemize}
\item if $z<0$, then
\[b((0,0),(z,x)) = \bot \neq \top;\]
\item if $z>0$, then
\[b((z,x),(0,0)) = \bot \neq \top;\]
\item if $z = 0$ with $x \neq 0$, then
\[b((z,x),(0,0)) = b((0,0),(z,x)) = \bot \neq \top.\]
\end{itemize}
Let us now consider the function $f \colon Y \rightarrow X$ defined, for any $(z,x) \in Y$, by
\[f(z,x) = x.\]
It is clear that $f$ is a group homomorphism. We now prove that, in addition, $f \colon (Y,b) \rightarrow (X,a)$ is a $V$-functor, i.e., for any $(z,x), (z',x') \in Y$,
\[b((z,x),(z',x')) \le a(f(z,x),f(z',x')),\]
i.e.
\[b((z,x),(z',x')) \le a(x,x').\]
Again, we consider the different possible cases:
\begin{itemize}
\item if $z<z'$, then
\[b((z,x),(z',x')) = a(x,x');\]
\item if $z=z'$ and $x=x'$, then
\[b((z,x),(z',x')) = \top = a(x,x) = a(x,x');\]
\item in the other cases,
\[b((z,x),(z',x)) = \bot \le a(x,x').\]
\end{itemize}
We can then conclude that $f$ is a $V$-homomorphism. It just remains to show that $f$ is an effective descent morphism in $V$-$\mathsf{Grp}$, which is the same as showing that it is a regular epimorphism. It is first of all clear that $f$ is surjective. Let us next prove that $f$ is also final, i.e. that, for any $x_1,x_2 \in X$,
\[a(x_1,x_2) = \bigvee_{f(y_i) = x_i} b(y_1,y_2).\]
We compute, for any $x_1,x_2 \in X$, that
\[\bigvee_{f(y_i) = x_i} b(y_1,y_2) = \bigvee_{z_1,z_2 \in \Z} b((z_1,x_1),(z_2,x_2)).\]
Accordingly,
\begin{itemize}
\item if $x_1=x_2$, then
\[\bigvee_{f(y_i) = x_i} b(y_1,y_2) = \top = a(x_1,x_2)\]
since $b((z,x_1),(z,x_2)) = \top$ for any $z \in \Z$;
\item if $x_1 \neq x_2$, then
\begin{align*}
\bigvee_{f(y_i) = x_i} b(y_1,y_2) & = \left( \bigvee_{z_1 < z_2} b((z_1,x_1),(z_2,x_2)) \right) \bigvee \left( \bigvee_{z_1 \ge z_2} b((z_1,x_1),(z_2,x_2)) \right)\\
& = a(x_1,x_2) \vee \bot\\
& = a(x_1,x_2).
\end{align*}
\end{itemize}
This completes the proof.
\end{proof}

\begin{proposition}
For any normal monomorphism $i \colon (K,b) \rightarrow (X,a)$ in $V$-$\mathsf{Grp}$, the monomorphism $i \cdot k \colon (N_K,\tilde{b})$ $\rightarrow (X,a)$ is normal, where $k \colon (N_K,\tilde{b}) \rightarrow (K,b)$ is the $(K,b)$-component of the counit of the coreflection $T \colon \text{$V$-}\mathsf{Grp} \rightarrow \text{$V$-}\mathsf{Grp_{ind}}$.
\end{proposition}

\begin{proof}
Since $i$ is a normal monomorphism, there exists an arrow $f \colon (X,a) \rightarrow (Z,c)$ in $V$-$\mathsf{Grp}$ such that $i = \mathsf{ker}(f)$. It means that $K$ is a normal subgroup of $X$ and that, for any $x,x' \in K$,
\[b(x,x') = a(i(x),i(x')).\]
Let us show that $N_K$ is a normal subgroup of $X$:
\begin{align*}
N_K & = \{x \in K \ | \ b(0,x) \geq k \ \text{and} \ b(x,0) \geq k\}\\
& = K \cap \{x \in X \ | \ a(0,x) \geq k \ \text{and} \ a(x,0) \geq k\}\\
& = K \cap N_X
\end{align*}
with $K$ and $N_X$ two normal subgroups of $X$. Hence, $N_K$ is normal in $X$. One then observes that, for any $n,n' \in N_K$,
\[a((i \cdot k)(n),(i \cdot k)(n')) = a(i(k(n)),i(k(n'))) = b(k(n),k(n')) = \tilde{b} (n,n').\]
We conclude that the monomorphism $i \cdot k$ is normal.
\end{proof}

It is now possible to apply Theorem \ref{thmEG}:

\begin{theorem} \label{description-coverings}
Let us consider the following classes of morphisms in $V$-$\mathsf{Grp}$:
\begin{itemize}
\item $\E' = \{f \in \text{$V$-}\mathsf{Grp} \ | \ \text{f is a normal epimorphism such that} \ \mathsf{Ker}(f) \in \text{$V$-}\mathsf{Grp_{ind}}\}$;
\item $\M^* = \{f \in \text{$V$-}\mathsf{Grp} \ | \ \mathsf{Ker}(f) \in \text{$V$-}\mathsf{Grp_{sep}}\}$.
\end{itemize}
Then, $(\E',\M^*)$ is a monotone-light factorization system.
\end{theorem}

\begin{proof}
This result follows from Theorem \ref{thmEG}, which can be applied to the reflection \eqref{reflection} thanks to the two previous propositions.
\end{proof}

As a consequence, the coverings in $V$-$\mathsf{Grp}$ with respect to the adjunction \eqref{reflection} are the $V$-homomorphisms $f \colon (X,a) \rightarrow (Y,b)$ such that $\mathsf{Ker}(f) \in$ $V$-$\mathsf{Grp_{sep}}$.

\section{Two observations}

We now make two observations related to the previous section. The first one concerns the coverings we have just characterized above, while the second one is about the torsion theory $(\text{$V$-}\mathsf{Grp_{ind}},\text{$V$-}\mathsf{Grp_{sep}})$ that we have in the category $V$-$\mathsf{Grp}$ when $V$ is an integral quantale.

\subsection*{Coverings of $V$-groups classified as internal actions}

We first remark that it is possible to classify the coverings in the category $V$-$\mathsf{Grp}$ of $V$-groups in terms of internal actions (whenever, as before, the quantale $V$ is integral). This can be done by means of a result of \cite{JMT} that we state below (see Theorem \ref{thm locally semisimple coverings}). For the reader's convenience, we remind a few notions needed for the understanding of that theorem. Note that, with respect to what is developed in \cite{JMT}, the context is slightly adapted in order to include the example of the category $V$-$\mathsf{Grp}$.  

Let $\C$ be any normal category in which normal epimorphisms and effective descent morphisms coincide. Let us consider a fixed class $\mathcal{X}$ of objects in $\C$, called a \textit{generalized semisimple class}, which is such that the following two properties hold for any pullback
\begin{center}
\begin{tikzcd}
E \times_{B} A \arrow[r,"\pi_{2}"] \arrow[d,"\pi_{1}"']
& A \arrow[d,"\alpha"]\\
E \arrow[r,"p"']
& B,
\end{tikzcd}
\end{center}
where $p$ is a normal epimorphism in $\C$:
\begin{enumerate}
\item $E \in \mathcal{X}$ and $A \in \mathcal{X}$ implies that $E \times_{B} A \in \mathcal{X}$;
\item $B \in \mathcal{X}$, $E \in \mathcal{X}$ and $E \times_{B} A \in \mathcal{X}$ implies that $A \in \mathcal{X}$.
\end{enumerate}
The notion of \textit{locally semisimple covering} is then defined relatively to a generalized semisimple class $\mathcal X$ in a category $\C$: a morphism $\alpha : A \rightarrow B$ is a locally semisimple covering in $\C$ when there is a normal epimorphism $p : E \rightarrow B$ such that the pullback $p^{*}(\alpha)$ of $\alpha$ along $p$ lies in the corresponding full subcategory $\mathcal X$ of $\C$. For a fixed $B \in \C$, we write $\mathsf{LocSSimple}_{\mathcal{X}}(B)$ for the full subcategory of the slice category $\C \downarrow B$ over $B$ whose objects are pairs $(A,\alpha)$, where $\alpha : A \rightarrow B$ is a locally semisimple covering. We then have the following result:

\begin{theorem} \cite{JMT} \label{thm locally semisimple coverings}
Consider a normal category $\C$ where normal epimorphisms are effective descent morphisms, and $\mathcal{X}$ a generalized semisimple class  in $\C$. If $p : E \rightarrow B$ is a normal epimorphism in $\C$ such that $E \in \mathcal{X}$, there exists an equivalence of categories
\[\mathsf{LocSSimple}_{\mathcal{X}}(B) \cong \mathcal{X}^{Eq(p)},\]
where $\mathcal{X}^{Eq(p)}$ denotes the full subcategory of $\C^{Eq(p)}$ whose objects are the internal $Eq(p)$-actions $(A_0,\pi,\xi)$ with $A_0 \in \mathcal{X}$.
\end{theorem}

The idea behind the proof of the above theorem is that, since normal epimorphisms are exactly effective descent morphisms in $\C$, $p$ is then an effective descent morphism so that the comparison functor
\[K_p \colon \C \downarrow B \rightarrow \C^{Eq(p)}\]
is an equivalence of categories. It is then possible to prove that, since $E \in \mathcal{X}$, when restricted to $\mathsf{LocSSimple}_{\mathcal{X}}(B)$, this equivalence corestricts to $\mathcal{X}^{Eq(p)}$ (see \cite{JMT} for a detailed proof).

We consider here the particular case where $\C = \text{$V$-}\mathsf{Grp}$ (with $V$ an integral quantale so that $V$-$\mathsf{Grp}$ is a normal category) and $\mathcal{X} = \text{$V$-}\mathsf{Grp_{sep}}$ which is easily seen (by using Lemma $2.7$ in \cite{GRossi}, for instance) to be a generalized semisimple class. By means of Theorem \ref{thm locally semisimple coverings} and also of the following lemma, we will see that it is then possible to characterize the coverings of $V$-groups as internal actions.

\begin{lemma}
A morphism $h \colon (Z,c) \rightarrow (X,a)$ in $V$-$\mathsf{Grp}$ is a locally semisimple covering (relatively to the subcategory $V$-$\mathsf{Grp_{sep}}$) if and only if its kernel is a separated $V$-group.
\end{lemma}

We refer to \cite{GM} for a proof of this lemma in the particular case of preordered groups since it is exactly the same idea for an arbitrary integral quantale $V$. 

\begin{remark}
\emph{The previous lemma is a particular case of a more general fact observed in \cite{JMT} (Proposition 2.3) where the role of the kernel of an arrow was played by the “fibers” (as defined in \cite{JMT}).}
\end{remark}

\begin{theorem} \label{thm action}
Let $(X,a) \in \text{$V$-}\mathsf{Grp}$. Consider the effective descent morphism $$f \colon (Y,b) \rightarrow (X,a)$$ from Proposition \ref{V-Grp_sep covers V-Grp}. Then there exists an equivalence of categories
\[ {\M}^* \downarrow(X,a) \cong \text{$V$-}\mathsf{Grp_{sep}}^{Eq(f)}\]
where ${\M}^* \downarrow(X,a)$  is the category of coverings over $(X,a)$.
\end{theorem}

\begin{proof}
Since the $V$-homomorphism $f \colon (Y,b) \rightarrow (X,a)$ from Proposition \ref{V-Grp_sep covers V-Grp} is an effective descent morphism such that $(Y,b) \in \text{$V$-}\mathsf{Grp_{sep}}$, by Theorem \ref{thm locally semisimple coverings}, we have the following equivalence of categories:
\[\mathsf{LocSSimple}_{\text{$V$-}\mathsf{Grp_{sep}}}(X,a) \cong \text{$V$-}\mathsf{Grp_{sep}}^{Eq(f)}.\]
The proof is now complete thanks to the previous lemma and Theorem \ref{description-coverings} since both the locally semisimple coverings and the coverings in $V$-$\mathsf{Grp}$ are given by the $V$-homomorphisms having their kernel in the subcategory of separated $V$-groups.
\end{proof}

\begin{remark}
\emph{By definition, the \emph{Galois groupoid} $\mathsf{Gal}(f)$ associated with $f \colon (Y,b) \rightarrow (X,a)$ (where, as before, $f$ denotes the effective descent morphism from Proposition \ref{V-Grp_sep covers V-Grp}) is the image of $Eq(f)$ by the reflector $F \colon \text{$V$-}\mathsf{Grp} \rightarrow \text{$V$-}\mathsf{Grp_{sep}}$. But since the diagram}
\begin{center}
\begin{tikzcd}
(Eq(f),b \wedge b) \arrow[r,shift left = 1.6] \arrow[r,shift right = 1.6]
& (Y,b)
\end{tikzcd}
\end{center}
\emph{lies in $V$-$\mathsf{Grp_{sep}}$, the image of $Eq(f)$ by $F$ is $Eq(f)$ itself. Accordingly, $Eq(f)$ coincides with the Galois groupoid associated with $f$, and}
\[ {\M}^* \downarrow(X,a) \cong \text{$V$-}\mathsf{Grp_{sep}}^{\mathsf{Gal}(f)}.\]
\emph{This equivalence \emph{classifies} the coverings in $V$-$\mathsf{Grp}$ as internal $\mathsf{Gal}(f)$-actions.}
\end{remark}

\subsection*{A pretorsion theory in $V$-$\mathsf{Grp}$}

Next, we show that the reflection \eqref{reflection} also gives rise to a pretorsion theory in the category $V$-$\mathsf{Grp}$ of $V$-groups. Note that, for this section, we do not need to assume that the quantale $V$ is integral.

\begin{proposition}
The pair of full and replete subcategories ($V$-$\mathsf{Grp_{sym}}$,$V$-$\mathsf{Grp_{sep}}$) of $V$-$\mathsf{Grp}$ is a $\z$-pretorsion theory in $V$-$\mathsf{Grp}$, where $\z =$ $V$-$\mathsf{Grp_{sym}}$ $\cap$ $V$-$\mathsf{Grp_{sep}}$ is given by
\[\z = \{(X,a) \in \text{$V$-}\mathsf{Grp} \ | \ a(x,x') = a(x',x) \ \forall x,x' \in X \ \text{and} \ a(x,x') \geq k \Longrightarrow x=x'\}.\]  
\end{proposition}

\begin{proof}
Let us first prove that any arrow
\[f \colon (X,a) \rightarrow (Y,b),\]
with $(X,a) \in$ $V$-$\mathsf{Grp_{sym}}$ and $(Y,b) \in$ $V$-$\mathsf{Grp_{sep}}$, belongs to $\n$. Consider then, in $V$-$\mathsf{Grp}$, the (regular epimorphism, monomorphism)-factorization of $f$:
\begin{center} 
\begin{tikzcd}
(X,a) \arrow[rr,"f"] \arrow[dr,two heads,"e"']
& & (Y,b)\\
 & (f(X),\bar{a}). \arrow[ur,tail,"m"'] &
\end{tikzcd}
\end{center}
Let us prove that $(f(X),\bar{a}) \in \z$. We first compute that, for any $y_1,y_2 \in f(X)$,
\[\bar{a}(y_1,y_2) = \bigvee_{f(x_i)=y_i} a(x_1,x_2) = \bigvee_{f(x_i) = y_i} a(x_2,x_1) = \bar{a}(y_2,y_1)\]
since $(X,a) \in$ $V$-$\mathsf{Grp_{sym}}$. Next, if we have $\bar{a}(y_1,y_2) \geq k$ for $y_1,y_2 \in f(X)$, then, since $(f(X),\bar{a}) \in$ $V$-$\mathsf{Grp_{sym}}$ and since $m$ is a $V$-functor,
\[k \le \bar{a}(y_1,y_2) = \bar{a}(y_1,y_2) \wedge \bar{a}(y_2,y_1) \le b(m(y_1),m(y_2)) \wedge b(m(y_2),m(y_1)),\]
which implies that
\[b(m(y_1),m(y_2)) \geq k \qquad \text{and that} \qquad b(m(y_2),m(y_1)) \geq k\]
with $m(y_1),m(y_2) \in Y$ and $(Y,b) \in$ $V$-$\mathsf{Grp_{sep}}$, so that $m(y_1) = m(y_2)$ and then $y_1 = y_2$. Accordingly, $f$ factorizes in $V$-$\mathsf{Grp}$ through an object of $\z$, and we have proved the first point of Definition \ref{def pretorsion theory}.

Consider now any $V$-group $(X,a)$ as well as the $V$-relation $\hat{a} \colon X \nvrightarrow X$ on $X$ defined, for any $x,x' \in X$, by
\[\hat{a}(x,x') = a(x,x') \wedge a(x',x).\]
Then, $(X,\hat{a})$ is a $V$-category since
\begin{itemize}
\item for any $x \in X$, $\hat{a}(x,x) = a(x,x) \wedge a(x,x) = a(x,x) \geq k$;
\item for any $x,x',x'' \in X$,
\[\hat{a}(x,x') \otimes \hat{a}(x',x'') = \left( a(x,x') \wedge a(x',x) \right) \otimes \left( a(x',x'') \wedge a(x'',x') \right),\]
with
\[\left( a(x,x') \wedge a(x',x) \right) \otimes \left( a(x',x'') \wedge a(x'',x') \right) \le a(x,x') \otimes a(x',x'') \le a(x,x'')\]
and
\begin{align*}
\left( a(x,x') \wedge a(x',x) \right) \otimes \left( a(x',x'') \wedge a(x'',x') \right) & \le a(x',x) \otimes a(x'',x') = a(x'',x') \otimes a(x',x)\\
& \le a(x'',x),
\end{align*}
so that 
\[\hat{a}(x,x') \otimes \hat{a}(x',x'') \le a(x,x'') \wedge a(x'',x) = \hat{a}(x,x'')\]
since $a(x,x'') \wedge a(x'',x)$ is by definition the biggest lower bound of $a(x,x'')$ and $a(x'',x)$.
\end{itemize}
It is moreover a $V$-group since $\hat{a}$ is invariant by shifting: for any $x,x',x'' \in X$,
\begin{align*}
\hat{a}(x',x'') & = a(x',x'') \wedge a(x'',x') \\
& = a(x'+x,x''+x) \wedge a(x''+x,x'+x)\\
& = \hat{a}(x'+x,x''+x).
\end{align*}
In addition, $(X,\hat{a})$ is in $V$-$\mathsf{Grp_{sym}}$. Indeed, for any $x,x' \in X$,
\[\hat{a}(x',x) = a(x',x) \wedge a(x,x') = a(x,x') \wedge a(x',x) = \hat{a}(x,x')\]
by commutativity of $\wedge$. Let us show that the sequence 
\begin{center}
\begin{tikzcd}
(X,\hat{a}) \arrow[r,"1_X"]
& (X,a) \arrow[r,"\eta_X"]
& (X/N_X,\bar{a})
\end{tikzcd}
\end{center}
is a short $\z$-exact sequence where, as before, we write
\begin{tikzcd}
X \arrow[r,two heads,"\eta_X"]
& X/N_X
\end{tikzcd}
for the quotient morphism, and where $(X/N_X,\bar{a})$ is a separated $V$-group, as already shown in the proof of Proposition \ref{torsion theory} (see Remark \ref{remark torsion theory}).

We begin by showing that $\eta_X$ is the $\z$-cokernel of the arrow $1_X$. Let us consider a morphism $f \colon (X,a) \rightarrow (Y,b)$ in $V$-$\mathsf{Grp}$ such that $f \cdot 1_X \in \n$, i.e. such that $f \cdot 1_X$ factorizes through an object $(Z,c)$ of $\z$: $f \cdot 1_X = \alpha \cdot \beta$.
\begin{center}
\begin{tikzcd}
(X,\hat{a}) \arrow[rr,"1_X"] \arrow[dr,"\beta"']
& & (X,a) \arrow[dr,"f"] \arrow[rr,two heads,"\eta_X"]
& & (X/N_X,\bar{a}) \arrow[dl,dotted,"\phi"]\\
 & (Z,c) \arrow[rr,"\alpha"']
& & (Y,b) 
&
\end{tikzcd}
\end{center}
Let $x \in N_X$. Then,
\[a(0,x) \geq k \qquad \text{and} \qquad a(x,0) \geq k\]
which implies that
\[k \le a(0,x) \wedge a(x,0) = \hat{a} (0,x) \le c(0,\beta(x)),\]
since $\beta$ is a $V$-homomorphism, so that
\[c(0,\beta(x)) \geq k.\]
It follows that $\beta(x) = 0$ since $(Z,c) \in \z$ and then that
\[f(x) = (\alpha \cdot \beta)(x) = 0\]
for any $x \in N_X$. Accordingly, by the universal property of the cokernel $\eta_X$, there exists a unique arrow $\phi \colon X/N_X \rightarrow Y$ in the category $\mathsf{Grp}$ of groups such that $\phi \cdot \eta_X = f$:
\begin{center}
\begin{tikzcd}
N_X \arrow[rr,tail,"k_X"]
& & X \arrow[rr,two heads, "\eta_X"] \arrow[dr,"f"']
& & X/N_X \arrow[dl, dotted, "\phi"]\\
 & & & Y. &
\end{tikzcd}
\end{center}
Let us show that this group morphism $\phi$ is in fact a $V$-functor. Consider $w_1, w_2 \in X/N_X$ and $x_1, x_2 \in X$ such that $\eta_{X}(x_i) = w_i$ for any $i = 1,2$. Then, we have that
\[a(x_1,x_2) \le b(f(x_1),f(x_2)) = b((\phi \cdot \eta_X) (x_1), (\phi \cdot \eta_X) (x_2)) = b(\phi(w_1), \phi(w_2))\]
since $f$ is a $V$-homomorphism, which implies that
\[\bar{a}(w_1,w_2) = \bigvee_{\eta_X(x_i) = w_i} a(x_1,x_2) \le b(\phi(w_1),\phi(w_2))\]
and this shows that the group morphism $\phi$ is a $V$-functor and so a $V$-homomorphism. As a conclusion, there exists a unique morphism $\phi \colon (X/N_X,\bar{a}) \rightarrow (Y,b)$ in $V$-$\mathsf{Grp}$ such that $\phi \cdot \eta_X = f$.

Next, we show that $1_X$ is the $\z$-kernel of $\eta_X$. Consider $f \colon (Y,b) \rightarrow (X,a)$ in $V$-$\mathsf{Grp}$ such that $\eta_X \cdot f \in \n$, i.e. $\eta_X \cdot f$ factorizes through an object $(Z,c)$ of $\z$: $\eta_X \cdot f = \alpha \cdot \beta$.
\begin{center}
\begin{tikzcd}
(X,\hat{a}) \arrow[rr,"1_X"]
& & (X,a) \arrow[rr,two heads,"\eta_X"]
& & (X/N_X,\bar{a})\\
 & (Y,b) \arrow[ul,dotted,"\phi"] \arrow[ur,"f"'] \arrow[rr,"\beta"']
& & (Z,c) \arrow[ur,"\alpha"']
&
\end{tikzcd}
\end{center}
Let us take $\phi = f$, since this is the only possible arrow in $\mathsf{Grp}$ such that $1_X \cdot \phi = f$. It remains to show that $\phi$ is a $V$-functor, i.e. that, for any $y,y' \in Y$,
\[b(y,y') \le \hat{a}(\phi(y),\phi(y')),\]
with 
\[\hat{a}(\phi(y),\phi(y')) = a(f(y),f(y')) \wedge a(f(y'),f(y)).\]
It is clear that
\[b(y,y') \le a(f(y),f(y'))\]
for any $y,y' \in Y$ since $f$ is a $V$-homomorphism. Let us then prove that we also have the following identity:
\[b(y,y') \le a(f(y'),f(y)).\]
Let $y,y' \in Y$. We compute that
\begin{align*}
b(y,y') & \le c(\beta(y),\beta(y')) = c(\beta(y'),\beta(y))\\
& \le \bar{a}((\alpha \cdot \beta)(y'),(\alpha \cdot \beta)(y))\\
& = \bar{a}((\eta_X \cdot f)(y'), (\eta_X \cdot f)(y)) = \displaystyle \bigvee_{\substack{\eta_X(x') = (\eta_X \cdot f)(y') \\ \eta_X(x) = (\eta_X \cdot f)(y)}} a(x',x)
\end{align*}
since $(Z,c) \in \z$. But, for any $x,x' \in X$ such that $\eta_X(x') = (\eta_X \cdot f)(y')$ and $\eta_X(x) = (\eta_X \cdot f)(y)$, we have that
\begin{align*}
a(f(y'),f(y)) & = a(f(y') - x' + x',f(y) - x + x)\\
& \ge a(f(y') - x',0) \otimes a(0,f(y) - x) \otimes a(x',x)\\
& \ge k \otimes k \otimes a(x',x) = a(x',x)
\end{align*}
since $+ \colon (X,a) \otimes (X,a) \rightarrow (X,a)$ is a $V$-functor and since $f(y') - x' \in N_X$ and $f(y) - x \in N_X$. As a consequence, for any $y,y' \in Y$,
\[b(y,y') \le \displaystyle \bigvee_{\substack{\eta_X(x') = (\eta_X \cdot f)(y') \\ \eta_X(x) = (\eta_X \cdot f)(y)}} a(x',x) \le a(f(y'),f(y)),\]
and this completes the proof: there exists a unique morphism $\phi \colon (Y,b) \rightarrow (X,\hat{a})$ in $V$-$\mathsf{Grp}$ such that $1_X \cdot \phi = f$.
\end{proof}

From this Proposition, we deduce in particular that the subcategory $V$-$\mathsf{Grp_{sym}}$ of symmetric $V$-groups is mono-coreflective in $V$-$\mathsf{Grp}$:

\begin{corollary}
The subcategory $V$-$\mathsf{Grp_{sym}}$ of symmetric $V$-groups is mono-coreflective in the category $V$-$\mathsf{Grp}$ of $V$-groups:
\begin{center} 
\begin{tikzcd}
\text{$V$-}\mathsf{Grp} \arrow[rr,shift right=2,"R"']
& \bot
& \text{$V$-}\mathsf{Grp_{sym}}. \arrow[ll,hook,shift right=2,"W"']
\end{tikzcd}
\end{center}
\end{corollary}

\begin{remark}
\emph{The coreflectivity of the subcategory $V$-$\mathsf{Grp_{sym}}$ of symmetric $V$-groups has already been observed and proved in \cite{CM} (Proposition 3.5). We mention this result here because we view it as a consequence of our pretorsion theory (which has not, to our knowledge, been discovered before) and also because we want to give the generalization of a proposition presented in \cite{GM} in the particular case of preordered groups. Note that Proposition 3.5 in \cite{CM} states that $V$-$\mathsf{Grp_{sym}}$ is reflective in $V$-$\mathsf{Grp}$ as well.}
\end{remark}

\begin{remark}
\emph{Similar results to the ones presented in this section can be obtained in the more general context of the category $V$-$\mathsf{Cat}$ of $V$-categories. We shall present these developments in a forthcoming article.}
\end{remark}

\end{document}